\documentclass[a4paper,12pt]{article}
\usepackage{amsmath,amsthm,amsfonts,amssymb}
\title{Compressions and Probably Intersecting Families}
\author{Paul A.~Russell
\footnote{Department
of Pure Mathematics and Mathematical Statistics,
Centre for Mathematical Sciences,
Wilberforce Road,
Cambridge CB3 0WB,
England.} 
\footnote{{\tt P.A.Russell@dpmms.cam.ac.uk}}}
\date{April 20,2011}

\newtheorem{theorem}{Theorem}
\newtheorem{lemma}{Lemma}
\newtheorem{corollary}[lemma]{Corollary}
\newtheorem{proposition}[lemma]{Proposition}

\newtheorem{problem}{Problem}
\newtheorem{question}[lemma]{Question}

\newcommand{\erdos}{Erd\H os}
\newcommand{\Wlog}{wlog}

\newcommand{\A}{{\cal A}}
\newcommand{\nr}{[n]^{(\gee r)}}
\newcommand{\nrr}{[n]^{(\gee r+1)}}
\newcommand{\size}{\sum_{j=r}^n{n\choose j}}
\newcommand{\Size}{\sum_{j=r+1}^n{n\choose j}}
\newcommand{\lee}{\leqslant}
\newcommand{\C}{{\cal C}}
\newcommand{\pow}{{\cal P}}
\newcommand{\gee}{\geqslant}
\newcommand{\ap}{\A_p}
\newcommand{\ner}{[n]^{(r)}}
\newcommand{\B}{{\cal B}}
\newcommand{\prob}{{\mathbb P}}
\newcommand{\ekr}{\erdos-Ko-Rado}
\newcommand{\II}{{\mathfrak I}}
\newcommand{\I}{{\mathfrak A}}
\newcommand{\J}{{\mathfrak C}}
\newcommand{\X}{{\cal X}}
\newcommand{\Y}{{\cal Y}}
\newcommand{\E}{{\cal E}}
\newcommand{\D}{{\cal D}}

\begin{document}

\maketitle

\begin{abstract}
A family $\A$ of sets is said to be {\it intersecting} if 
$A\cap B\ne\emptyset$ for all $A$, $B\in\A$.  It is a well-known and simple 
fact that an intersecting family of subsets of $[n]=\{1,2,\ldots\,,n\}$
can contain at most $2^{n-1}$ sets. 
Katona, Katona and Katona ask the following question.
Suppose instead $\A\subset\pow[n]$ satisfies $|\A|=2^{n-1}+i$ for some fixed
$i>0$.  Create a new family $\ap$ by choosing each member of $\A$
independently with some fixed probability $p$.  How do we choose $\A$ to
maximize the probability that $\ap$ is intersecting?
They conjecture that there is a nested sequence of optimal families for
 $i=1$, $2$, $\ldots\,$, $2^{n-1}$.  In this paper,
we show that the families $\nr=\{A\subset[n]:|A|\ge r\}$ are optimal for
the appropriate values of $i$, thereby proving the conjecture for this
sequence of values.
Moreover, we show that for intermediate values of
$i$ there exist optimal families lying between those we have found.  
It turns out
that the optimal families we find simultaneously maximize the number of
intersecting subfamilies of every possible order.

Standard compression techniques appear inadequate to solve the problem
as they do not preserve intersection properties of subfamilies.  Instead,
our main tool is a 
novel compression method, together with a way of `compressing'
subfamilies,  which may be of independent interest.
\end{abstract}

\begin{section}{Introduction}\label{intro}
Many problems of extremal combinatorics concern intersecting familes of
finite sets.  A family $\A$ is said to be {\it intersecting} if $A\cap B
\ne\emptyset$ for all $A$,~$B\in\A$.  How large an
intersecting family can we find in the discrete cube $Q_n=\pow[n]
=\pow\{1,2,\ldots\,,n\}$?  It is easy to achieve $|\A|=2^{n-1}$,
for example by taking $\A=\{A\subset Q_n:1\in A\}$.  And it is easy to
see that we can do no better than this---an intersecting family cannot
contain both a set and its complement.

A more interesting question arises if we require our intersecting family
to be {\it uniform.}  Given a set $S$ and a positive integer $r$,
write $S^{(r)}$ for the collection $\{A\subset S:|A|=r\}$ of all subsets of
$S$ of size $r$.  How large an intersecting family $\A\subset\ner$ can we
find?

As in the non-uniform case, it seems natural to try taking
\hbox{$\A=\{A\in\ner:1\in A\}$,} 
here achieveing $|\A|={n-1\choose r-1}$.  And indeed,
in their significant paper of 1964, \erdos, Ko and Rado \cite{ekr} show
that if $r\lee n/2$ we can do no better than this.  (We remark in passing
that the problem is of no interest if $r>n/2$, as then the entirety
of $\ner$ is itself intersecting.)

In this paper we shall be concerned with two related probabilistic questions
posed by Katona, Katona and Katona \cite{kkk}.  We begin with the non-uniform
case.

Recall from above that if $\A\subset\pow[n]$ is intersecting then
$|\A|\lee2^{n-1}$.  Suppose that we are instead required to choose a somewhat
larger family $\A$ and then randomly discard some of the sets in $\A$ to
form a subfamily $\B$.  How can we maximize the probability that $\B$ is
intersecting?  A precise statement of the problem is as follows.

\begin{problem}[\cite{kkk}]\label{pnon}
Let $n$ and $i$ be positive integers with $i\lee 2^{n-1}$ and let $p\in(0,1)$.
Given $\A\subset\pow[n]$, write $\ap$ for the (random) subfamily of $\A$
obtained by choosing each set in $\A$ independently with probability $p$.
How should we choose $\A$ with $|\A|=2^{n-1}+i$ to maximize
$\prob(\ap\hbox{\rm\ is intersecting})$?
\end{problem}

Katona, Katona and Katona~\cite{kkk} solve the first cases of this problem,
that is, for $i\lee {n-1\choose\lceil(n-3)/2\rceil}$.  They construct their
optimal families by taking `large' subsets of the cube.  More precisely,
for $n$ odd take all sets of size at least $(n+1)/2$ and any $i$ sets
of size $(n-1)/2$ that contain the element $1$.  Similarly, for $n$ even
take all sets of size $n/2+1$, all sets of size $n/2$ that contain the element
$1$, and any other $i$ sets of size $n/2$.  They conjecture that a
continuation of this construction gives an optimal family $\A$ for each $i$,
leading to a nested sequence 
$\A_1\subset\A_2\subset\cdots\A_{2^{n-1}}$ of optimal families
for $i=1$, $2$, $\ldots\,$, $2^{n-1}$.

In this paper, we show that the families $\nr=\{A\subset\pow[n]:|A|\ge r\}$
are optimal for the appropriate values of $i$, thereby proving the conjecture
for this sequence of values.
Moreover, we show that for
intermediate values of $i$ there exist optimal families lying between
those we have found.  Our main result is
as follows.  

\begin{theorem}\label{tnon}
Let $n$ be a positive integer and $p\in(0,1)$.  Let $r$ be a positive
integer with $r\le n/2$.  Then, over all $\A\subset\pow[n]$
with $|\A|=\size$, the probability $\prob(\ap\hbox{\rm\ is intersecting})$
is maximized by $\A=\nr$.

Moreover, suppose $i$ is any positive integer with $i\lee 2^{n-1}$ and let
$r$ be such that 
$\Size\lee2^{n-1}+i\lee\size$.  Then, over all $\A\subset\pow[n]$ with
$|A|=2^{n-1}+i$, the probability $\prob(\ap\hbox{\rm\ is intersecting})$
is maximized by some $\A$ with $\nrr\subset\A\subset\nr$.
\end{theorem}

We remark that the result of Theorem~\ref{tnon} is independent of the
value of $p$.  In fact, for $2^{n-1}+i=\size$, the family $\nr$ simultaneously
maximizes the number of intersecting subfamilies of each possible order.
This result may be of independent interest.

We also consider the uniform version of the problem.

\begin{problem}[\cite{kkk}]\label{puni}
Let $n$, $r$ and $i$ be positive integers with $r\lee n/2$ and
\hbox{$i\lee{n-1\choose r}$,} and let $p\in(0,1)$.  How should we choose
$\A\subset[n]^{(r)}$ with \hbox{$|\A|={n-1\choose r-1}+i$} to maximize
$\prob(\ap\hbox{\rm\ is intersecting})$?
\end{problem}

Results on this problem seem rather harder to come by:  Katona, Katona
and Katona~\cite{kkk} solve only the first case $i=1$.  Using methods similar
to those used to prove Theorem~\ref{tnon}, we show
that, for each $i$, there is an optimal family that is left-compressed
(as explained below).  Unfortunately, our methods are not sufficient to
determine which amongst the left-compressed families of given order is
best.  

\begin{theorem}\label{tuni}
Let $n$, $r$ and $i$ be positive integers with $r\lee n/2$ and
$i\lee{n-1\choose r}$, and let $p\in(0,1)$.  Then there exists a
left-compressed family $\A\subset\ner$ with $|\A|={n-1\choose r-1}+i$
that maximizes $\prob(\ap\hbox{\rm\ is intersecting})$ over all subfamilies
of $\ner$ of order ${n-1\choose r-1}+i$.
\end{theorem}

Many fruitful approaches to intersection problems involve the use of
compression techniques, first introduced by \erdos\, Ko and Rado~\cite{ekr}
in the proof of their uniform intersection theorem mentioned above.  The
idea behind such techniques is that, starting from an intersecting family
$\A$, one `moves' certain sets in $\A$ to make $\A$ `nicer' in some way
whilst $\A$ retains the property of being intersecting.  The proof
of the \ekr\ theorem applies $ij$-compressions, defined as follows.

Let $i$, $j\in[n]$ with $i<j$.  If $A\in\ner$ then the 
{\it $ij$-compression of $A$} is
$$C_{ij}A=\left\{\begin{array}{cl}
(A\cup\{i\})-\{j\}&\hbox{if }j\in A,i\not\in A\\
A&\hbox{otherwise}\end{array}\right..$$
If $\A\subset\ner$, the {\it $ij$-compression of $\A$} is
$$\C_{ij}\A=\{C_{ij}A:A\in\A\}\cup\{A\in\A:C_{ij}A\in\A\}.$$
Informally, we replace $j$ by $i$ whenever we can.  We may be prevented from
replacing $j\in A$ by $i$ {\sc either} because $i$ is already in $A$
{\sc or} because $C_{ij}A$ is already in $\A$.  When we replace $j\in A$
by $i$, we say that $A$ {\it moves;} that is, $A$ moves if $j\in A$, 
$i\not\in A$ and $C_{ij}A\not\in\A$.  
We say that $A$ is {\it blocked} from moving by $C_{ij}A$ if $A\ne C_{ij}A$
and $C_{ij}A\in\A$.
A family $\A$ is {\it $ij$-compressed}
if $\A=\C_{ij}\A$.  It is {\it left-compressed} if it is $ij$-compressed
whenever $i<j$.

\erdos, Ko and Rado show that if $\A\subset\ner$ is intersecting then so
is $\C_{ij}\A$.  They also check that any $\A\subset\ner$ can be transformed
to a left-compressed family by repeated $ij$-compressions.  It hence suffices
for them to consider only left-compressed families in their proof.

It seems at first that a similar approach to Problems~\ref{pnon} 
and~\ref{puni}
of Katona, Katona and Katona cannot possibly succeed.  We know 
from~\cite{ekr} that compressing an intersecting family yields an intersecting
family.  Unfortunately, if we compress a non-intersecting family $\A$
then there may exist an intersecting subfamily of $\A$ which moves to a
non-intersecting subfamily of $\C_{ij}\A$.

Here is a simple example which illustrates the main obstacle.  Consider
applying a $12$-compression to the family $\A=\{13,23,24\}$.  Only $24$ moves,
giving $\C_{12}\A=\{13,23,14\}$.  But now $\B=\{23,24\}\subset\A$ is
intersecting and moves to $\{23,14\}$ which is not.  (What has gone wrong?
The set $23$ was blocked from moving by the set $13$ which is in $\A$
but not in $\B$.)

Nevertheless, 
we are able to show that the family $\C_{ij}\A$ has at least
as many intersecting subfamilies of each given order as does the family
$\A$.  In fact, there is a fairly natural injection $\phi$ from the collection
$\I$ of intersecting subfamilies of $\A$ to the collection $\J$ of intersecting
subfamilies of $\C_{ij}\A$.  Starting from an intersecting family $\B\in\I$,
we form the family $\phi(\B)$ by replacing appropriately chosen $B\in\B$
by $C_{ij}B$.  We must obviously choose to replace those $B\in\B$ that move
when $\A$ is compressed to $\C_{ij}\A$, as in this case $B\not\in\C_{ij}\A$.
But we also choose to replace certain $B\in\B$ that were blocked from moving
by $C_{ij}B\in\A$ but for which $C_{ij}B\not\in\B$.  The choice of which
such $B$ to replace depends both on the family $\A$ and the subfamily $\B$.
In \S\ref{left} we give the details of our construction and prove that the
resulting families $\phi(\B)$ are indeed intersecting as required.  This
will establish Theorem~\ref{tuni}.

Our launching pad for Theorem \ref{tuni} was the use of $ij$-compressions
to prove the \ekr\ theorem.  Can we find something to play a similar role
for Theorem~\ref{tnon}?  The right place to start turns out to be from a
more general compression operator first introduced by Daykin~\cite{daykin} in
his beautiful proof of the Kruskal-Katona theorem (\cite{kruskal}, 
\cite{katona}).  
These ``$UV$-compressions'' were independently discovered by Frankl and
F\"uredi~\cite{ff} in their proof of Harper's theorem.
They also turn out to be
a special case of a compression operator later
developed by Bollob\'as and Leader~\cite{bl}, who use them to prove
intersection theorems such as the \erdos-Ko-Rado theorem and Katona's 
$t$-intersecting theorem~\cite{supervision}.  This proof of the 
$t$-intersecting theorem was also found independently by Ahlswede and 
Khachatrian~\cite{ak}.

When attempting to apply these methods to Problem~\ref{pnon},
the same obstacle arises as in the proof of Theorem~\ref{tuni} and is
overcome in the same way.  However, further difficulties arise in this case.
To preserve intersection properties in the proof of the $t$-intersecting
theorem, it is necessary to carry out the $UV$-compressions in a carefully
chosen order.  But even when this is done, we are unable to show that the
number of intersecting subfamilies of each order increases whenever an 
individual $UV$-compression is applied.  
Instead, it appears that we must apply a
sequence of several $UV$-compressions together, after which there are at
least as many intersecting subfamilies of each order as before.
We shall explain this further in 
\S\ref{up}, where we prove Theorem~\ref{tnon}.

Finally, in \S\ref{end}, we make some concluding remarks and mention some
open problems.

Our notation is mostly standard.  We draw the reader's attention
to certain points.
We write $[n]$ for the set $\{1,2,\ldots\,,n\}$ and $[m,n]$ for the
set $\{m,m+1,\ldots\,,n\}$.  For any set $S$, we write $S^{(r)}$ for the
set $\{A\subset S:|A|=r\}$ of all subsets of $S$ of order $r$,
and $S^{(\gee r)}$ for the set $\{A\subset S:|A|\gee r\}$ of all subsets of
$S$ of order at least $r$.  
If $X$ and $Y$ are sets we write $X-Y$ for the set $\{x\in X:x\not\in Y\}$.
For ease of reading, we often omit set brackets
and union symbols.  Thus, for example, $123$ denotes the set $\{1,2,3\}$,
$12XY$ denotes the set $\{1,2\}\cup X\cup Y$, and $1X\cap Y$ denotes the
set $(\{1\}\cup X)\cap Y$.  If $\A$ is a family of sets, we write
$\II(\A)$ for the collection of all intersecting subfamilies of $\A$; that
is, 
$\II(\A)=\{\B\subset\A:\B\hbox{ is intersecting}\}$.
\end{section}

\begin{section}{Left-compression}\label{left}
Our aim in this section is to prove Theorem~\ref{tuni}.

Let $i,j\in[n]$ with $i<j$.  Recall from \S\ref{intro} the definition
of the {\it $ij$-compression.}
If $A\in\ner$ then the 
{\it $ij$-compression of $A$} is
$$C_{ij}A=\left\{\begin{array}{cl}
(A\cup\{i\})-\{j\}&\hbox{if }j\in A,i\not\in A\\
A&\hbox{otherwise}\end{array}\right..$$
If $\A\subset\ner$, the {\it $ij$-compression of $\A$} is
$$\C_{ij}\A=\{C_{ij}A:A\in\A\}\cup\{A\in\A:C_{ij}A\in\A\}.$$

It is easy to see that for any $\A\subset\ner$ we may obtain a left-compressed
family by applying an appropriate sequence of $ij$-compressions.  (For example,
the quantity $\sum_{A\in\A}\sum_{a\in A}a$ decreases whenever we apply
a non-trivial $ij$-compression.)  So it suffices to prove that if
$\C=\C_{ij}\A$ then $\prob(\C_p\hbox{ is intersecting})\gee
\prob(\A_p\hbox{ is intersecting})$.  This will follow immediately from the
following lemma which is the heart of the proof.

\begin{lemma}\label{lij}
Let $\A\subset\ner$, let $i$, $j\in[n]$ and let $\C=\C_{ij}\A$.
Then there exists an injection $\phi\colon\II(\A)\to\II(\C)$ such that
$|\phi(\B)|=|\B|$ for all $\B\in\II(\A)$.
\end{lemma}

\begin{proof}
Assume \Wlog\ $i=1$ and $j=2$.  Write $\I=\II(\A)$ and $\J=\II(\C)$.  Let
\begin{eqnarray*}
\A_1&=&\{X\subset[3,n]:1X\in\A,2X\not\in\A\}\\
\A_2&=&\{X\subset[3,n]:1X\not\in\A,2X\in\A\}\\
\A_{12}&=&\{X\subset[3,n]:1X,2X\in\A\}\\
\A_0&=&\{X\in\A:1,2\in X\hbox{ or }1,2\not\in X\}.
\end{eqnarray*}
Observe that $\A$ may be written as the disjoint union
$$\A=\{1X:X\in\A_1\cup\A_{12}\}\cup\{2X:X\in\A_2\cup\A_{12}\}\cup\A_0.$$
We make similar definitions and a similar observation for the family $\C$.
We have $\C_1=\A_1\cup\A_2$, $\C_2=\emptyset$, $\C_{12}=\A_{12}$ and
$\C_0=\A_0$.

Suppose $\X=(\X_1,\X_2,\X_{12,(0)},\X_{12,(1)},\X_{12,(2)},\X_0)$
where $\X_1\subset\A_1$, $\X_2\subset\A_2$, $\X_0\subset\A_0$ and
$\X_{12,(0)}$, $\X_{12,(1)}$, $\X_{12,(2)}$ form a disjoint partition
of $\A_{12}$.  Let $\I_\X\subset\I$ be the collection of intersecting
families $\B\subset\A$ satisfying the following conditions:
\newcounter{bean}
\begin{list}{(\roman{bean})}{\usecounter{bean}}
\item for $X\in\A_1$, $1X\in\B\iff X\in\X_1$;
\item for $X\in\A_2$, $2X\in\B\iff X\in\X_2$;
\item for $X\in\A_0$, $X\in\B\iff X\in\X_0$;
\item for $X\in\A_{12}$: 
\begin{itemize}
\item if $X\in\X_{12,(0)}$ then $1X$, $2X\not\in\B$;
\item if $X\in\X_{12,(1)}$ then $1X\in\B$ or $2X\in\B$ but not both;
\item if $X\in\X_{12,(2)}$ then $1X$, $2X\in\B$.
\end{itemize}
\end{list}
Let $\J_\X\subset\J$ be the collection of intersecting families $\B\subset\C$
satisfying condtions (i), (iii) and (iv), and the additional condition:
\begin{list}{(ii)'}{\usecounter{bean}}
\item for $X\in\A_2$, $1X\in\B\iff X\in\X_2$.
\end{list}

Observe that $\I$ and $\J$ can be written as disjoint unions
$\I=\bigcup_\X\I_\X$ and $\J=\bigcup_\X\J_\X$, the union in each case
ranging over all permissible values of $\X$.  Moreover, for each $\X$
there is a positive integer $m$ such that $|\B|=m$ for every
$\B\in\I_\X\cup\J_\X$.  Hence it suffices to construct, for each $\X$,
an injection $\phi_\X\colon\I_\X\to\J_\X$.

So fix $\X$.  Let 
$$\Y=\{X\in\X_{12,(1)}:2X\in\B\hbox{ for all }\B\in\I_\X\}.$$  Define
$\phi_\X\colon\I_\X\to\J_\X$ by
$$\phi_\X(\B)=\big(\B\cup\{1X:X\in\X_2\cup\Y\}\big)-\{2X:X\in\X_2\cup\Y\}.$$
In order to check that $\phi_\X$ is well-defined, we must check that
$\phi_\X(\B)$ is intersecting for each $\B\in\I_\X$.  It will then be
clear that $\phi_\X$ is an injection from $\I_\X$ to $\J_\X$.

Assume for a contradiction that $\B\in\I_\X$ but that $\D=\phi_\X(\B)$
is not intersecting.  So there are sets $A$, $B\in\D$ with $A\cap B=\emptyset$.
As $\B$ is intersecting, we cannot have both $A$, $B\in\B$, so assume
\Wlog\ $A\not\in\B$.  Then $A=1X$ for some $X\in\X_2\cup\Y$ and $2X\in\B$.
Now we must have $B\in\B$ (as otherwise we would have $1\in B$).  So
$B\cap 2X\ne\emptyset$ but $B\cap 1X=\emptyset$.  Hence $B=2Y$ for some
$Y\subset[3,n]$ with $X\cap Y=\emptyset$.  We cannot have $1Y\in\B$
(as $1Y\cap 2X=\emptyset$) so $Y\in\X_2\cup\X_{12,(1)}$.  But,
as $2Y\in\D$, we have $Y\not\in\X_2\cup\Y$.  So $Y\in\X_{12,(1)}-\Y$;
that is, there is some $\E\in\I_\X$ with $1Y\in\E$.  But $X\in\X_2\cup\Y$
and so $2X\in\E$.  But $1Y\cap 2X=\emptyset$, a contradiction (as
$\E$ is intersecting).

Thus $\phi_\X$ is an injection from $\I_\X$ to $\J_\X$ for each $\X$.
Putting together all of the $\phi_\X$, we obtain the required injection
$\phi\colon\II(\A)\to\II(\C)$.
\end{proof}

We immediately obtain the main result of this section.

\setcounter{theorem}{1}
\begin{theorem}
Let $n$, $r$ and $i$ be positive integers with $r\lee n/2$ and
$i\lee{n-1\choose r}$, and let $p\in(0,1)$.  Then there exists a
left-compressed family $\A\subset\ner$ with $|\A|={n-1\choose r-1}+i$
that maximizes $\prob(\ap\hbox{ is intersecting})$ over all subfamilies
of $\ner$ of order ${n-1\choose r-1}+i$.
\end{theorem}

\begin{proof}
Let $\A\subset\ner$ be a family of order $k$ maximizing $\prob(\A_p\hbox{ is
intersecting})$ over all families of order $k$.  Carry out a sequence
of $ij$-compressions $\C_{i_1j_1}$, $\C_{i_2j_2}$, $\ldots\,$, $\C_{i_mj_m}$
to obtain families $\A_1=\C_{i_1j_1}\A$, $\A_2=\C_{i_2j_2}\A_1$, 
$\ldots\,$, \hbox{$\A_m=\C_{i_mj_m}\A_{m-1}$,} with $\A_m$ left-compressed.

It follows from Lemma~\ref{lij} that, for any family $\B$ and any $i$, $j$,
we have $\prob((\C_{ij}\B)_p\hbox{ is intersecting})\gee\prob(\B_p\hbox{ is
intersecting})$.  Hence, by induction, $\prob((\A_m)_p\hbox{ is intersecting})
\gee\prob(\A_p\hbox{ is intersecting})$.  But the family
$\A$ was chosen to maximize this
probability, so in fact we have that  
\hbox{$\prob((\A_m)_p\hbox{ is intersecting})=\prob
(\A_p\hbox{ is intersecting})$} and $\A_m$ is our required
left-compressed family.
\end{proof}

\end{section}

\begin{section}{Main result}\label{up}
We now turn to the proof of our main result, Theorem~\ref{tnon}.  As we
remarked in \S\ref{intro}, we begin from a certain proof of Katona's
$t$-intersecting theorem using $UV$-compressions.  In \S\ref{UV} we
define $UV$-compressions, briefly outline this proof of the $t$-intersecting
theorem, and explain where
the difficulties lie in translating these methods to solve Problem~\ref{pnon}.
In \S\ref{Uvf} we define our new compression operators.  Finally, in
\S\ref{main} we prove Theorem~\ref{tnon}.

\begin{subsection}{Background}\label{UV}
Let $n$ be a positive integer and let $U$, $V\subset[n]$ be disjoint.
If $A\subset[n]$ then the {\it $UV$-compression of $A$} is
$$C_{UV}A=\left\{\begin{array}{cl}
(A\cup U)-V&\hbox{if }V\subset A, U\cap A=\emptyset\\
A&\hbox{otherwise}\end{array}\right..$$
If $\A\subset\pow[n]$, the {\it $UV$-compression of $\A$} is
$$\C_{UV}\A=\{C_{UV}A:A\in\A\}\cup\{A\in\A:C_{UV}A\in\A\}.$$
As with $ij$-compressions, it is generally helpful to think of the compression
`moving' certain sets by replacing $V$ with $U$ where possible.  Indeed,
$ij$-compressions are simply the special case of $UV$-compressions where
$U$ and $V$ are both singleton sets.

Again as with $ij$-compressions, a typical application aims to compress
an initial family to make it `nicer' in some way whilst preserving some
property of the family.  However, one must often take great care
over the order in which the compressions are applied.

A well known example is the $t$-intersecting theorem.  A family
$\A\subset\pow[n]$ is said to be {\it $t$-intersecting} if
$|A\cap B|\gee t$ for all $A$, $B\in\A$.  How large can such a family be?

Assume for simplicity that $n+t$ is even.  One obvious example is to take
$\A=[n]^{\left(\gee{n+t\over 2}\right)}$.  
Katona~\cite{tint} showed that this was
best possible.  We sketch a later proof based on 
$UV$-compressions.

The proof begins with a $t$-intersecting family $\A$ and aims to transform
it into a family $\B$ with $\nrr\subset\B\subset\nr$.  This can be done
by a sequence of $UV$-compressions with $|V|<|U|$ in each case.  (In fact,
we need only use $UV$ compressions with $|U|=|V|+1$.)  If the
resulting family $\B$ is $t$-intersecting then the theorem is proved.

Unfortunately, this need not always be the case:  the family
$\A=\{45,46\}$ is $1$-intersecting but $\C_{123,45}\A=\{123,46\}$
is not.  However, this problem can be resolved by carrying out the simplest
available compression at each stage---here $\A$ is not $(12,4)$-compressed,
and $\C_{12,4}\A=\{125,126\}$ {\it is} $1$-intersecting. (To be precise,
it is now easy to check that 
if $\A$ is $t$-intersecting and $U'V'$-compressed for all $U'\subset U$
and $V'\subset V$ with $|U'|>|V'|$ and $(U',V')\ne(U,V)$ then 
$\C_{UV}\A$ is $t$-intersecting.)
This suffices to prove the $t$-intersecting theorem.

We now consider how this can be applied to Problem~\ref{pnon}.  We begin
with a family $\A\subset\pow[n]$ which we aim to compress to a family
$\A'$ with \hbox{$\nrr\subset\A'\subset\nr$}
 by means of $UV$-compressions with
$|U|=|V|+1$.  Our initial hope might be that if these compressions are
applied in an appropriate order then, as with $ij$-compressions in 
\S\ref{left}, the number of intersecting subfamilies of each possible order
increases after each compression.

We may clearly apply a $UV$-compression with $|V|=0$---each intersecting
subfamily of $\A$ moves to an intersecting subfamily of $\C_{UV}\A$.

If $|V|=1$ then, as with $ij$-compressions, it is possible for an
intersecting subfamily of $\A$ to move to a non-intersecting
subfamily of $\C_{UV}\A$.  But this problem can be resolved precisely
as it was for $ij$-compressions in the proof of Lemma~\ref{lij}.

The real problem first arises when $|V|=2$.
Now we are unable to show that $\C_{UV}\A$ contains more intersecting
subfamilies of each order than does $\A$, even if we assume that we have
already performed all simpler compressions (although we do not have a
counterexample).

Why does the proof of Lemma~\ref{lij} not carry over?  
Suppose, say, we perform the compression $\C_{123,45}$ on $\A$, and
$\B\subset\A$ is intersecting.  Perhaps when forming $\phi(\B)$ we replace
$45\in\B$ with $123$.  Now, if also $4\in\B$ then $4$ does not move but
$4\cap123=\emptyset$.  However, we know that $\A$ is $(12,4)$-compressed
so maybe we can replace $4$ with $12$.  But what if, say, $34\in\B$?  Now
$34$ does not intersect $12$ \ldots

At some point in the proof, it appears that we need to perform an
illegal replacement, say $34\to125$.  And we cannot assume that $\A$ is
$(125,34)$-compressed as then we do not obtain a well-founded order in which
to carry out our compressions.

The solution is to perform four compressions together---instead of comparing
$\A$ with $\C_{123,45}\A$, we compare it with 
$\C=\C_{123,45}\C_{125,34}\C_{134,25}\C_{145,23}\A$.  It is now possible
to arrange that all of the necessary replacements are legal, yielding a proof
that $\C$ contains at least as many intersecting subfamilies of each possible
order as does $\A$.
\end{subsection}

\begin{subsection}{$(U,v,f)$-compressions}\label{Uvf}
It is convenient to define a new compression operator which carries out
all of the necessary compressions simultaneously.  In fact, it moves sets
in such a way that we no longer need to worry about carrying out simpler
compressions first.

Let $X$ be a set.  A {\it pairing function} on $X$ is a function
$f\colon X\to X$ such that $f\circ f$ is the identity and $f$ has no fixed
point.  We may think of $f$ as `pairing' the elements of $X$.  Note that
if $X$ is finite then it must have even order.

Let $U\subset[n]$ be of even order, $v\in[n]-U$ and $f\colon U\to U$ be
a pairing function.  We define the $(U,v,f)$-compression on $\pow[n]$
by
$$C_{U,v,f}(A)=\left\{\begin{array}{cl}
A&\hbox{if }v\in A\\
f(A\cap U)\cup\{v\}\cup(A-U)&\hbox{if }v\not\in A
\end{array}\right.$$
for $A\in\pow[n]$
and
$$\C_{U,v,f}(\A)=\{C_{U,v,f}(A):A\in\A\}\cup\{A\in\A:C_{U,v,f}(A)\in\A\}$$
for $\A\subset\pow[n]$.

We remark that in the case where $\A$ is already $U'V'$-compressed for all 
disjoint pairs $(U',V')$ with
$V'\subset U$, $U'\subset U\cup\{v\}$, \hbox{$|V'|<|U|/2$,} 
\hbox{$|U'|\lee|U|/2+1$}
and $|U'|>|V'|$ then $\C_{U,v,f}$ can be written as a
composition of $UV$ compressions.  Indeed, in this case
\hbox{$C_{U,v,f}=C_{U_1V_1}C_{U_2V_2}\cdots C_{U_kV_k}$} where $V_1$, $V_2$,
$\ldots\,$, $V_k$ are the subsets of $U$ of order $|U|/2$ and 
\hbox{$U_i=(U-V_i)\cup\{v\}$.}

As an illustration, we prove that
$(U,v,f)$-compressions preserve the property of a family being intersecting.

\begin{proposition}
Let $\A\subset\pow[n]$ be intersecting, let $U\subset[n]$ be of even order,
let $v\in[n]-U$ and let $f\colon U\to U$ be a pairing function.
Then $\C_{U,v,f}\A$ is intersecting.  
\end{proposition}

\begin{proof}
Write $\C=\C_{U,v,f}\A$.  Suppose that $\C$ is not intersecting.  Choose
$A$, $B\in\C$ with $A\cap B=\emptyset$.
As $\A$ is intersecting, we cannot
have both $A$, $B\in\A$, so assume \Wlog\ that $A\not\in\A$.  Then 
$A=vf(W)X$ for some $W\subset U$ and $X\subset [n]-(U\cup\{v\})$ with
$WX\in\A$.  As $A\cap B=\emptyset$, we must have $v\not\in B$ and thus
$B\in\A$ and $B=TY$ for some $T\subset U$ and $Y\subset [n]-(U\cup\{v\})$.
Moreover, $T\cap f(W)=\emptyset$ and $X\cap Y=\emptyset$.  Now,
as $v\not\in B$ and $B\in\C$ we must have $vf(T)Y\in\A$.  Now consider
$WX$, $vf(T)Y\in\A$.  We have $W\cap f(T)=f(f(W)\cap T)=\emptyset$ and 
$X\cap Y=\emptyset$, and so $WX\cap vf(T)Y=\emptyset$.  But this is a
contradiction, as $\A$ is intersecting.
\end{proof}

Note that, unlike with standard $UV$-compressions, 
there is no restriction here on the order in which these
compressions may be applied.  However, we remark in passing that the order
{\it would} be important if we wanted to retain the property of $\A$
being $2$-intersecting; in this case we would again have to apply
compressions with smaller $U$ first.  For example, taking $\A=\{23,1236\}$,
$U=\{2345\}$, $v=1$, $f(2,3,4,5)=(4,5,2,3)$, we have $\A$ $2$-intersecting
but $C_{U,v,f}\A=\{145,1236\}$ not $2$-intersecting.
\end{subsection}

\begin{subsection}{Proof of main result}\label{main}
The heart of the proof is the following lemma.  The proof of the lemma
mirrors the proof of Lemma~\ref{lij}, but using our $(U,v,f)$-compressions
in place of $ij$-compressions.

\begin{lemma}\label{luvf}
Let $\A\subset\pow[n]$, let $U$, $v$ and $f$ be as above and let 
$\C=\C_{U,v,f}(\A)$.  Then there exists an injection 
$\phi\colon\II(\A)\to\II(\C)$ such that
$|\phi(\B)|=|\B|$ for all $\B\in\II(\A)$.
\end{lemma}

\begin{proof}
Write $\I=\II(\A)$, $\J=\II(\C)$ and $S=[n]-(U\cup\{v\})$.  For each
$W\subset U$, let
\begin{eqnarray*}
\A_1^W&=&\{X\subset S:WX\not\in\A,vf(W)X\in\A\}\\
\A_2^W&=&\{X\subset S:WX\in\A,vf(W)X\not\in\A\}\\
\A_{12}^W&=&\{X\subset S:WX,vf(W)X\in\A\}.
\end{eqnarray*}
Observe that $\A$ may be written as the disjoint union
$$\A=\bigcup_{W\subset U}\left(\{vf(W)X:X\in\A_1^W\cup\A_{12}^W\}
\cup\{WX:X\in\A_2^W\cup\A_{12}^W\}\right).$$
We make similar definitions and a similar observation for the family $\C$.
For each $W\subset U$ we have $\C_1^W=\A_1^W\cup\A_2^W$, $\C_2^W=\emptyset$
and $\C_{12}^W=\A_{12}^W$.

Suppose 
$\X=(\X_1^W,\X_2^W,\X_{12,(0)}^W,\X_{12,(1)}^W,\X_{12,(2)}^W)_{W\subset U}$
where, for each \hbox{$W\subset U$,} we have $\X_1^W\subset\A_1^W$, 
$\X_2^W\subset\A_2^W$ and $\X_{12,(0)}^W$, $\X_{12,(1)}^W$ and 
$\X_{12,(2)}^W$ forming a disjoint partition of $\A_{12}^W$.
Let $\I_\X\subset\I$ be the collection of intersecting families $\B\subset\A$
satisfying, for each $W\subset U$, the following conditions:
\begin{list}{(\roman{bean})}{\usecounter{bean}}
\item for $X\in\A_1^W$, $vf(W)X\in\B\iff X\in\X_1^W$;
\item for $X\in\A_2^W$, $WX\in\B\iff X\in\X_2^W$;
\item for $X\in\A_{12}^W$:
\begin{itemize}
\item if $X\in\X_{12,(0)}^W$ then $WX$, $vf(W)X\not\in\B$;
\item if $X\in\X_{12,(1)}^W$ then $WX\in\B$ or $vf(W)X\in\B$ but not both;
\item if $X\in\X_{12,(2)}^W$ then $WX$, $vf(W)X\in\B$.
\end{itemize}
\end{list}
Let $\J_\X$ be the collection of intersecting families $\B\subset\C$
satisfying, for each $W\subset U$, conditions (i) and (iii) and the additional
condition
\begin{list}{(ii)'}{\usecounter{bean}}
\item for $X\in\A_2^W$, $vf(W)X\in\B\iff X\in\A_2^W$.
\end{list}
Observe that $\I$ and $\J$ can be written as disjoint unions 
$\I=\bigcup_\X\I_\X$ and $\J=\bigcup_\X\J_\X$, and that, for each $\X$,
there is a positive integer $m$ such that $|\B|=m$ for all 
$B\in\I_\X\cup\J_X$.  Hence, as before, it suffices to construct, for each
$\X$, an injection $\phi_\X\colon\I_\X\to\J_\X$.

So fix $\X$.  For each $W\subset U$ let
$$\Y^W=\{X\in\X_{12,(1)}^W:WX\in\B\hbox{ for all }\B\in\I_\X\}.$$
Define $\phi_\X\colon\I_\X\to\J_\X$ by
$$\phi_\X(\B)=\B\cup\bigcup_{W\subset U}\{vf(W)X:X\in\X_2^W\cup\Y^W\}
-\bigcup_{W\subset U}\{WX:X\in\X_2^W\cup\Y^W\}.$$
Again, all that we need to check is that $\phi_\X(\B)$ is intersecting
for each $\B\in\I_\X$.

Assume for a contradiction that $\B\in\I_\X$ but that $\D=\phi_\X(\B)$
is not intersecting.  So there are sets $A$, $B\in\D$ with 
$A\cap B=\emptyset$.  As $\B$ is intersecting, we cannot have both $A$,
$B\in\B$ so assume \Wlog\ $A\not\in\B$.  Then $A=vf(W)X$ for some
$W\subset U$, $X\in\X_2^W\cup\Y^W$ and $WX\in\B$. 
Now, we must have $B\in\B$ (as otherwise we would have
$v\in B$ and so $A\cap B\ne\emptyset$).  
So $B\cap WX\ne\emptyset$ but $B\cap vf(W)X=\emptyset$.
Hence $B=TY$ for some $T\subset U$ and $Y\subset[n]-(U\cup\{v\})$
with $T\cap f(W)=\emptyset$ and $X\cap Y=\emptyset$.

It is easy to check that $vf(T)\cap W=\emptyset$.  Indeed, suppose instead
that there is some $a\in vf(T)\cap W$.  As $v\not\in W$ we must have $a\ne v$
and so $a=f(t)$ for some $t\in T$.  But then $t=f(a)\in f(W)$,
contradicting $T\cap f(W)=\emptyset$.

Now, we have $vf(T)\cap W=\emptyset$ and $X\cap Y=\emptyset$, so
$vf(T)Y\cap WX=\emptyset$.  But $WX\in\B$ so $vf(T)Y\not\in\B$.  Now,
$vf(T)Y\not\in\B$ but $TY\in\B$ so $Y\in\X_2^T\cup\X_{12,(1)}^T$.  But
$TY=B\in\D$ so $Y\not\in\X_2^T\cup\Y^T$.  Hence $Y\in\X_{12,(1)}^T-\Y^T$;
that is, there is some $\E\in\I_\X$ with $vf(T)Y\in\E$.

But $vf(W)X=A\in\D$ and $vf(W)X\not\in\B$ so $X\in\X_2^W\cup\Y^W$.  Hence
$WX\in\E$.  But now $vf(T)Y$, $WX\in\E$ with $vf(T)Y\cap WX=\emptyset$,
a contradiction.
\end{proof}

We now obtain our main result.

\vfill
\eject

\setcounter{theorem}{0}
\begin{theorem}
Let $n$ be a positive integer and $p\in(0,1)$.  Let $r$ be a positive
integer with $r\le n/2$.  Then, over all $\A\subset\pow[n]$
with $|\A|=\size$, the probability $\prob(\ap\hbox{\rm\ is intersecting})$
is maximized by $\A=\nr$.

Moreover, suppose $i$ is any positive integer with $i\lee 2^{n-1}$ and let
$r$ be such that 
$\Size\lee2^{n-1}+i\lee\size$.  Then, over all $\A\subset\pow[n]$ with
$|A|=2^{n-1}+i$, the probability $\prob(\ap\hbox{\rm\ is intersecting})$
is maximized by some $\A$ with $\nrr\subset\A\subset\nr$.
\end{theorem}

\begin{proof}
It clearly suffices to prove the second statement as the first follows 
immediately.
Starting from any family $\A$, we observe that 
it is possible to obtain a family $\C$ with
$\nrr\subset\C\subset\nr$ for some $r$ by a sequence of $(U,v,f)$-compressions.
Indeed, suppose that $\A$ is not already of the required form.  Then is
is easy to see that there are some disjoint sets $W$, $V\in[n]$ with
$|W|=|V|+1$ and $\A$ not $WV$-compressed.  Take $v$=$\min W$,
$U=(W-\{v\})\cup V$ and $f\colon U\to U$ a pairing function with
$f(W-\{v\})=V$.  Then $\A$ is not $(U,v,f)$-compressed so we may apply
$\C_{U,v,f}$ to obtain a new family.  But every time we apply a 
non-trivial $(U,v,f)$-compression the quantity $\sum_{A\in\A}|A|$ increases
and so this process must terminate with some $\A$ of the required form.
Hence Theorem~\ref{tnon} follows from Lemma~\ref{luvf} precisely as
Theorem~\ref{tuni} follows from Lemma~\ref{lij}.
\end{proof}

Examining the proof of Lemma~\ref{luvf}, we note that $\C_{U,v,f}\A$ has
at least as many intersecting subfamilies of each possible order as does
$\A$.  Hence our optimal families simultaneously
maximize the number of intersecting subfamilies of every possible order.
This may be of independent interest.

\begin{corollary}\label{sim}
Let $\A\subset\pow[n]$ with
$|A|=\sum_{j=r}^n{n\choose r}$.  
Then the family $\nr$ has
at least as many intersecting subfamilies of every possible order 
as has $\A$.
\end{corollary} 

In fact, Theorem~\ref{tnon} also solves Problem~\ref{pnon} in the cases.
where \hbox{$2^{n-1}+i=\big(\sum_{j=r}^n{n\choose r}\big)\pm1$.}  
Moreover, Lemma~\ref{lij}
holds for $\A\subset\pow[n]$ as well as for $\A\subset\ner$ with an identical
proof, allowing us to solve Problem~\ref{pnon} in the cases where
$2^{n-1}+i=\big(\sum_{j=r}^n{n\choose r}\big)\pm2$.  
For completeness, we state these
results explicitly.

\begin{corollary}\label{next}
Let $n$ and $r$ be positive integers with $r\le n/2$.
\begin{itemize}
\item
Over all $\A\subset\pow[n]$ with
$|\A|=\big(\sum_{j=r}^n{n\choose r}\big)+1$, 
the probability $\prob(\ap\hbox{\rm\ is intersecting})$
is maximized by $$\A=\nr\cup\{123\ldots r-1\}.$$

\item Over all $\A\subset\pow[n]$ with 
$|\A|=\big(\sum_{j=r}^n{n\choose r}\big)-1$.
the probability $\prob(\ap\hbox{\rm\ is intersecting})$ is maximized
by $$\A=\nr-\{(n-r+1)(n-r+2)\ldots n\}.$$

\item
Over all $\A\subset\pow[n]$ with
$|\A|=\big(\sum_{j=r}^n{n\choose r}\big)+2$, 
the probability $\prob(\ap\hbox{\rm\ is intersecting})$
is maximized by $$\A=\nr\cup\{123\ldots r-1,123\ldots (r-2)r\}.$$

\item Over all $\A\subset\pow[n]$ with 
$|\A|=\big(\sum_{j=r}^n{n\choose r}\big)-2$,
the probability $\prob(\ap\hbox{\rm\ is intersecting})$
is maximized by $$\A=\nr-\{(n-r)(n-r+2)(n-r+3)\ldots n, 
(n-r+1)(n-r+2)\ldots n\}.$$
\end{itemize}
Moreover, in each case the optimal family simultaneously
maximizes the number of intersecting subfamilies of every possible order.
\end{corollary}

\end{subsection}

\end{section}

\begin{section}{Concluding remarks}\label{end}
Theorem~\ref{tnon} gives us substantial information about the structure
of the optimal families solving Problem~\ref{pnon}:  they consist of the top
layers of the cube together with some collection of sets from the next layer
down.  However, we know little about what happens within the layers.  The
analogue of Lemma~\ref{lij} for families in $\pow[n]$ gives that we may
take our optimal family to be left-compressed, but this still leaves open
many possibilities.  In particular, we would be interested to know if there
is indeed a nested sequence of optimal families as conjectured by Katona,
Katona and Katona~\cite{kkk}.

We observed in Corollaries~\ref{sim} and~\ref{next} that in all the cases
of Problem~\ref{pnon} that we could solve, the optimal family simultaneously
maximized the number of intersecting subfamilies of every given order.  We
would like to know if this is always possible.

\begin{question}
Let $N\lee 2^n$.  Does there exist a family $\A\subset\pow[n]$ of order $N$
which simultaneously
maximizes the number of intersecting subfamilies of every possible order?
\end{question}

Finally, we noted in \S\ref{UV} that we were unable to prove that 
$UV$-compressions (applied in appropriate order) always increased the number
of intersecting subfamilies of each order.  However, we also have no
counterexample.  Hence we ask:

\begin{question}
Let $\A\subset\pow[n]$ and $U$, $V\subset[n]$ be disjoint with $|U|>|V|$.
Suppose $\A$ is $U'V'$-compressed for all $U'\subset U$ and $V'\subset V$
with $(U',V')\ne(U,V)$ and $|U'|>|V'|$.  Must $\C_{UV}\A$ have at least
as many intersecting subfamilies of every possible order as has $\A$?
\end{question}
\end{section}

\end{document}